\documentclass[letterpaper, 10pt, dvipsnames]{article}

\usepackage{bold-extra}
\usepackage[utf8]{inputenc}    
\usepackage[T1]{fontenc}
\usepackage[english]{babel}
\usepackage[normalem]{ulem}
\usepackage{dsfont, bm, pifont, mathrsfs}
\usepackage{stmaryrd}
\usepackage{comment}
\usepackage{bbm}
\usepackage{enumitem}
\usepackage{float, graphicx, caption, wrapfig, tikz}
\usepackage[margin=30pt]{subcaption}

\usepackage{amsthm, amsmath, amsfonts, amssymb, mathrsfs, mathtools}
\usepackage[alphabetic]{amsrefs}

\usepackage{hyperref, xcolor, titlesec} 
\hypersetup{colorlinks = true, urlcolor = RoyalBlue!70!Black,linkcolor=BrickRed, citecolor=BlueViolet, bookmarksopen = true}
\usepackage[nomarginpar]{geometry}
\numberwithin{equation}{section}

\usepackage{cleveref}

\DeclareMathOperator{\Var}{Var}

\let \d \relax
\newcommand{\d}{\mathrm{d}}

\newcommand{\E}{\mathds{E}}
\newcommand{\unn}[2]{[\![#1,#2]\!]}

\renewcommand{\epsilon}{\varepsilon}
\renewcommand{\tilde}{\widetilde}

\def\bet{\begin{thm}}
\def\eet{\end{thm}}
\def\bel{\begin{lem}}
\def\eel{\end{lem}}
\def\bas{\begin{ass}}
\def\eas{\end{ass}}
\def\bec{\begin{cor}}
\def\eec{\end{cor}}
\def\bed{\begin{defn}}
\def\eed{\end{defn}}
\def\bep{\begin{prop}}
\def\eep{\end{prop}}
\def\beq{\begin{equation}}
\def\eeq{\end{equation}}
\def\bea{\begin{equation*}}
\def\eea{\end{equation*}}
\def\bex{\begin{ex}}
\def\eex{\end{ex}}
\def\bp{\begin{proof}}
\def\ep{\end{proof}}

\def\1{{\mathbbm 1}}

\def\benr{\begin{enumerate}[label=(\roman*)]}
\def\eenr{\end{enumerate}}

\def\R{\mathbb{R}}

\newcommand{\bma}{\begin{bmatrix}}
\newcommand{\ema}{\end{bmatrix}}

\def\phi{\varphi}

\renewcommand{\hat}{\widehat}

\newtheorem{ccounter}{ccounter}[section]
\newtheorem{thm}[ccounter]{Theorem}
\newtheorem{lem}[ccounter]{Lemma}
\newtheorem{cor}[ccounter]{Corollary}
\newtheorem{defn}[ccounter]{Definition}
\newtheorem{prop}[ccounter]{Proposition}
\newtheorem{ass}[ccounter]{Assumption}
\newtheorem{ex}[ccounter]{Example}

\theoremstyle{definition}

\titleformat{\paragraph}[runin]{\itshape\normalsize}{\theparagraph}{}{}
\titleformat{\subparagraph}[runin]{\itshape\normalsize}{\theparagraph}{0em}{}
\titleformat{\section}[block]{\normalfont\filcenter}{\Large\thesection .}{.7em}{\Large\scshape}
\titleformat{\subsection}[runin]{\normalfont}{\large \bf \thesubsection .}{.5em}{\large\bf}[.]
\titleformat{\subsubsection}[runin]{\normalfont}{\bf \thesubsubsection .}{.5em}{\bf}[.]

\usepackage{tocloft}

\begin{document}

\title{\textsc{\textbf{Eigenvalue distribution of the Hadamard product of sample covariance matrices in a quadratic regime}}}
\author{S. \textsc{Abou Assaly}\\\vspace{-0.15cm}\footnotesize{\it{Université de Montr\'eal}}\\\footnotesize{\it{sebastien.abou.assaly@umontreal.ca}}\and L. \textsc{Benigni}\\\vspace{-0.15cm}\footnotesize{\it{Université de Montréal}}\\\footnotesize{\it{lucas.benigni@umontreal.ca}}}

\date{}
\maketitle
\abstract{In this note, we prove that if $X\in\R^{n\times d}$ and $Y\in\R^{n\times p}$ are two independent matrices with i.i.d entries then the empirical spectral distribution of $\frac{1}{d}XX^\top \odot \frac{1}{p}YY^\top$, where $\odot$ denotes the Hadamard product, converges to the Marchenko--Pastur distribution of shape $\gamma$ in the quadratic regime of dimension $\frac{n}{dp}\to \gamma$ and $\frac{p}{d}\to a$.}
\section{Introduction and main results}\label{intro}

First studied by Wishart \cite{wishart1928generalised} for some problems in statistics, random matrix theory has then skyrocketed after the works of Wigner \cite{wigner} in mathematical physics and Marchenko and Pastur \cite{marchenko} in statistics which both considered high-dimensional problems. In the latter, the authors considered the empirical eigenvalue distribution of matrices of the form $\frac{1}{d}XX^\top$ where $X\in\R^{n\times d}$ is a matrix with i.i.d centered entries with unit variance (the paper actually concerns more general cases where the entries of $X$ have some covariance structure). They then show that, 
\[
\frac{1}{n}\sum_{i=1}^n \delta_{\lambda_i}\xrightarrow[n\to\infty]{}\mu_{\mathrm{MP}_\gamma} 
\]  
weakly in probability if $n$ and $d$ grows together linearly in the sense that $\frac{n}{d}\to \gamma$. If $\gamma\in(0,1]$, the Marchenko--Pastur distribution is given by 
\begin{equation}\label{eq:mp}
\mu_{\mathrm{MP}_\gamma}(\d x) = \frac{\sqrt{(b-x)(x-a)}}{2\pi\gamma x}\mathds{1}_{x\in[a,b]}\d x\quad\text{with }b=(1+\sqrt{\gamma})^2\quad\text{and }a = (1-\sqrt{\gamma})^2.
\end{equation}
If $\gamma>1$, since the eigenvalues of $XX^\top$ and $X^\top X$ are the same up to a mass at 0, it is enough to add $(1-\frac{1}{\gamma})\delta_0$ to the measure. This result has been generalized to many different models and different assumptions on the entries of $X$, see \cite{baisilver} for a book on the subject.

Recently, paradigms in machine learning have introduced new models of random matrix theory involving different structures or different scaling of dimensions, see \cite{rmmlbook} for a book on the subject. For instance, the analysis of large neural networks have developed further the study of matrices where a nonlinearity is applied entrywise \cites{louartliaocouillet, pennington2017nonlinear, benignipeche1, fan2020spectra, piccolo2021analysis, wangzhu,dabo2024traffic}, models which were first introduced through inner-product kernel matrices \cites{elkaroui, chengsinger}. Beside nonlinear random matrices, other forms of structures have risen. The Neural Tangent Kernel corresponds to the covariance matrix of the Jacobian of the output of a neural network during training \cites{jacot2018neural,chizat2019lazy} and can be written, in the simpler case of a two-layer network, in the form 
\begin{equation}\label{eq:NTK}
\mathrm{NTK} = \frac{1}{d}XX^\top \odot \frac{1}{p}\sigma'(XW)D^2\sigma'(XW)^\top +\frac{1}{p}\sigma(XW)\sigma(XW)^\top
\end{equation}
where $\sigma:\R\to\R$, $X\in\R^{n\times d},$ $W\in\R^{d\times p}$, $D\in\R^{p\times p}$ represents respectively the activation function, the matrix of data, the matrix of weights of the hidden layer at initialization, the diagonal matrix of the output layer vector at initialization, and $\odot$ is the Hadamard or entrywise product $(A\odot B)_{ij}=A_{ij}B_{ij}$. The spectrum of this matrix has been studied in \cites{fan2020spectra, wangzhu} in a regime of dimension where the first part collapses to the identity matrix. The regime where this part becomes nontrivial is given by $n\asymp dp$ and has been studied in \cite{benignipaquette}. We note that such polynomial scalings are becoming important for these models as training data samples are becoming huge and has been studied in numerous recent works \cites{misiakiewicz2022spectrum, lu2022equivalence, dubova2023universality, hu2024asymptotics, montanarizhong, pandit2024universality, misiakiewicz2023six}.

The literature on the Hadamard product of random matrices is sparse and this sparks the subject of the current paper. We consider a simpler model than \eqref{eq:NTK} where the two matrices being multiplied entrywise are independent and are both null sample covariance matrices.  
We consider the distribution of the eigenvalues of 
\begin{equation}\label{eq:defM}
M = \frac{1}{d} XX^\top \odot \frac{1}{p} YY^\top
\end{equation}
 with $\odot$ being the Hadamard product, $X\in\R^{n\times d}$ with i.i.d entries, and $Y\in\R^{n\times p}$ with i.i.d entries independent of $X$. We give the following assumptions on the distribution of $X$ and $Y$
\begin{ass}\label{ass1}
We have that there exists $C>0$, such that for all $i\in\unn{1}{n},$ $j\in\unn{1}{d}$, and $\ell\in\unn{1}{p}$
\begin{itemize}
    \item $\E\left[X_{ij}\right]=0$,\, $\E\left[X_{ij}^2\right]=\sigma_x^2$, and $\E\left[X_{ij}^4\right]\leqslant C$
    \item $\E\left[Y_{i\ell}\right]=0$,\, $\E\left[Y_{i\ell}^2\right]=\sigma_y^2$, and $\E\left[Y_{i\ell}^4\right]\leqslant C$
\end{itemize}
\end{ass}
These assumptions are quite general. However, we do not know whether the existence of a fourth moment is necessary for the convergence of the empirical eigenvalue distribution to the Marchenko--Pastur distribution. We have the following assumptions on the scaling of dimensions
\begin{ass}\label{ass2}
    We suppose that
    \[
    \frac{n}{pd}\xrightarrow[n\to\infty]{}\gamma >0,\quad \frac{p}{d}\xrightarrow[n\to\infty]{}a>0. 
    \]
\end{ass}
Our main result is given by the following 
\begin{thm}\label{theo}
    Under Assumptions \ref{ass1} and \ref{ass2}, the empirical eigenvalue distribution $\mu_n$ of $\frac{1}{\sigma_x^2\sigma_y^2}M$ defined in \eqref{eq:defM} converges weakly in probability to the Marchenko–Pastur distribution of shape $\gamma$ defined in \eqref{eq:mp}. 
    
    Additionally, if we assume that all moments of $X$ and $Y$ are finite, we obtain almost sure weak convergence of the empirical eigenvalue distribution.
    \end{thm}
This result is illustrated in Figure \ref{fig:fig1}.
\begin{figure}[!ht]
    \centering
    \includegraphics[width=.5\linewidth]{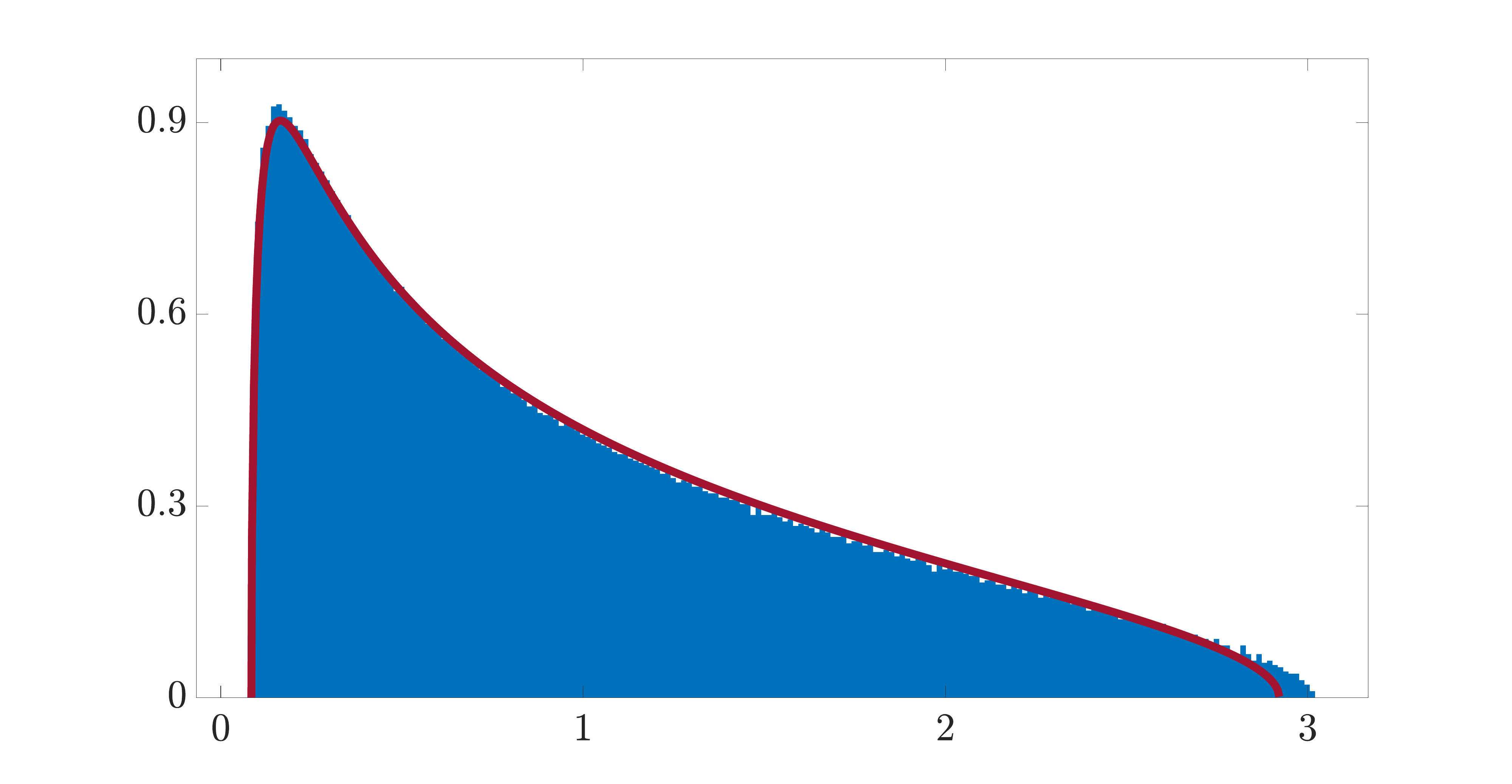}
    \caption{Histogram of eigenvalues of $M$ for $n=20000$, $p=282$, $d=141$ with the curve of the Marchenko--Pastur distribution.}\label{fig:fig1}
\end{figure}
\section{Proof of the main result}

The proof is based on the method of moments. We start by obtaining the limiting expected moments in Proposition \ref{prop1} and then prove concentration of such moments in Proposition \ref{prop2}. This will be proved assuming that all moments of $X$ and $Y$ are finite. To come back to Assumption \ref{ass1}, we use the truncation and centralisation technique in Proposition \ref{prop3}. We note that other proofs are possible notably with the Stieltjes transform of the empirical eigenvalue distribution. This can be done using the result from \cite{baizhou} which consists in checking concentration of quadratic forms.

\begin{prop}\label{prop1}

Suppose that we have the assumptions made in Theorem \ref{theo} and that for all $k\in\mathbb{N}$, $\E \left[\vert X_{ij}\vert^k\right]+\E\left[\vert Y_{i\ell}\vert^k\right] < +\infty$ for $i\in\unn{1}{n},$ $j\in\unn{1}{d}$, and $\ell\in\unn{1}{p}$. Then,

    \[\E \left[ \int x^{k}\, d\mu_n \right] \xrightarrow[n\to\infty]{} \sum_{s=0}^{k-1}  \frac{\gamma^s}{s+1} {\binom{k-1}{s}} {\binom{k}{s}}\]

\end{prop}

\begin{proof}
We start by calculating
\[
\mu_n^k \coloneqq \E \left[ \int x^{k}\, d\mu_n \right]
=
\E \left[ \int x^{k} \cdot \frac{1}{n} \sum_{i=1}^{n} \delta_{\lambda_i}\, dx \right] 
= 
\frac{1}{n} \E \left[ \sum_{i=1}^{n} \int x^{k} \delta_{\lambda_i}\, dx \right]
= 
\frac{1}{n} \E \left[ \sum_{i=1}^{n} \lambda_i^{k} \right].
\]
In addition, by the definition of the trace and the following property $\text{Tr}(AB)=\text{Tr}(BA)$, we have:
\[
\mu_n^k
= 
\frac{1}{n} \E \left[ \text{Tr}(M^k) \right]
= \frac{1}{n} \E \left[ \sum_{1 \leq i_1, \ldots, i_k \leq n} M_{i_1 i_2} M_{i_2 i_3} \ldots M_{i_k i_1} \right]
\]
using the definition of $M$,
\begin{align*}
\mu_n^k
=
\frac{1}{n}  \sum_{1 \leq i_1, \ldots, i_k \leq n}\! \E \left[ \frac{1}{dp} \left[ XX^\top \right]_{i_1 i_2} \left[ YY^\top \right]_{i_1 i_2}  \frac{1}{dp} \left[ XX^\top \right]_{i_2 i_3} \left[ YY^\top \right]_{i_2 i_3} \ldots \frac{1}{dp} \left[ XX^\top \right]_{i_k i_1} \left[ YY^\top \right]_{i_k i_1} \right] 
\\
=
\frac{1}{n d^k p^k}  \sum_{\substack{1 \leq i_1, ..., i_k \leq n \\ 1 \leq m_1, ..., m_k \leq d \\ 1 \leq j_1, ..., j_k \leq p}} \E \left[X_{i_1 m_1} X_{i_2 m_1} X_{i_2 m_2} X_{i_3 m_2} \ldots X_{i_k m_k} X_{i_1 m_k} Y_{i_1 j_1} Y_{i_2 j_1} Y_{i_2 j_2} Y_{i_3 j_2} \ldots Y_{i_k j_k} Y_{i_1 j_k} \right].
\end{align*}
We now introduce the notations
\[
I \vcentcolon = (i_1, i_2, \ldots, i_k),\quad M \vcentcolon = (m_1, m_2, \ldots, m_k), \quad J \vcentcolon = (j_1, j_2, \ldots, j_k)
\]
and 
\[
X_{IM} \vcentcolon = X_{i_1 m_1} X_{i_2 m_1} \ldots X_{i_k m_k} X_{i_1 m_k},\quad  Y_{IJ} \vcentcolon = Y_{i_1 j_1} Y_{i_2 j_1} \ldots Y_{i_k j_k} Y_{i_1 j_k}.
\]
We can then write
\[
\mu_n^k
= 
\frac{1}{n d^k p^k} \sum_{I \in \unn{1}{n}^k}\sum_{M \in \unn{1}{d}^k}\sum_{J \in \unn{1}{p}^k} \E \left[ X_{IM} Y_{IJ} \right]
\]

Suppose that $W_x \vcentcolon = (i_1, m_1, i_2, m_2, \ldots, i_k, m_k, i_1)$ is the path made by the vertices $i_s$ and $m_s$ and $W_y \vcentcolon = (i_1, j_1, i_2, j_2, \ldots, i_k, j_k, i_1)$ is the path created by the vertices $i_s$ and $j_s$ for $s\in\unn{1}{k}$. We denote $\mathcal{G}$ the set of all the graphs $G$ with 4$k$ steps that have the same shape as the paths $\left( W_x, W_y \right)$. This is illustrated in Figure \ref{fig:fig2}. We also define $\Pi(G)$ as the expectation of $X_{IM} Y_{IJ}$ knowing the shape of the graph $G$, hence we can write,
\[
\mu_n^k
= 
\frac{1}{n d^k p^k} \sum_{G \in \mathcal{G}} \Pi(G) \cdot \# \{(W_x,W_y): G \text{ has the shape of } (W_x,W_y)\}.
\]

\begin{figure}[!ht]
\centering
\begin{tikzpicture}[node distance = 80 ,main/.style = {draw, circle, minimum size=12,}]

\node[main, label = left:{$i_1$},] (1) at (0,0) {};
\node[main, label = below:{$m_1 = m_5$}, fill = RoyalBlue] (2) at (1.25, -1.25) {};
\node[main, label = above:{$j_1 = j_5$}, fill = BrickRed] (3) at (1.25, 1.25) {};
\node[main, label = {[label distance=.2em]left:{$i_2 = i_5$}}] (4) at (2.5, 0) {};
\node[main, label = below:{$m_2 = m_4$}, fill = RoyalBlue] (5) at (3.75, -1.25) {};
\node[main, label = above:{$j_2 = j_4$}, fill = BrickRed] (6) at (3.75, 1.25) {};
\node[main, label = {[label distance=.2em]left:{$i_3 = i_4$}}] (7) at (5, 0) {};
\node[main, label = below:{$m_3$}, fill = RoyalBlue] (8) at (6.25, -1.25) {};
\node[main, label = above:{$j_3$}, fill = BrickRed] (9) at (6.25, 1.25) {};

\draw[->, RoyalBlue, line width = .1em] (1.-25) -- (2.115);
\draw[->, RoyalBlue, line width = .1em] (2.155) -- (1.-65);
\draw[->, RoyalBlue, line width = .1em] (2.65) -- (4.205);
\draw[->, RoyalBlue, line width = .1em] (4.245) -- (2.25);
\draw[->, RoyalBlue, line width = .1em] (4.-25) -- (5.115);
\draw[->, RoyalBlue, line width = .1em] (5.155) -- (4.-65);
\draw[->, RoyalBlue, line width = .1em] (5.65) -- (7.205);
\draw[->, RoyalBlue, line width = .1em] (7.245) --  (5.25);
\draw[->, RoyalBlue, line width = .1em] (7.-25) -- (8.115);
\draw[->, RoyalBlue, line width = .1em] (8.155) -- (7.-65);

\draw[->, BrickRed, line width = .1em] (1.65) -- (3.-155);
\draw[->, BrickRed, line width = .1em] (3.-115) -- (1.25);
\draw[->, BrickRed, line width = .1em] (3.-25) -- (4.115);
\draw[->, BrickRed, line width = .1em] (4.155) -- (3.-65);
\draw[->, BrickRed, line width = .1em] (4.65) -- (6.-155);
\draw[->, BrickRed, line width = .1em] (6.-115) -- (4.25);
\draw[->, BrickRed, line width = .1em] (6.-25) -- (7.115);
\draw[->, BrickRed, line width = .1em] (7.155) -- (6.-65);
\draw[->, BrickRed, line width = .1em] (7.65) -- (9.-155);
\draw[->, BrickRed, line width = .1em] (9.-115) -- (7.25);

\end{tikzpicture}
\caption{Example of a graph where $W_x = (i_1,m_1,i_2,m_2,i_3,m_3,i_3,m_2,i_2,m_1,i_1)$ and the other path is given by $W_y = (i_1,j_1,i_2,j_2,i_3,j_3,i_3,j_2,i_2,j_1,i_1).$}\label{fig:fig2}
\end{figure}
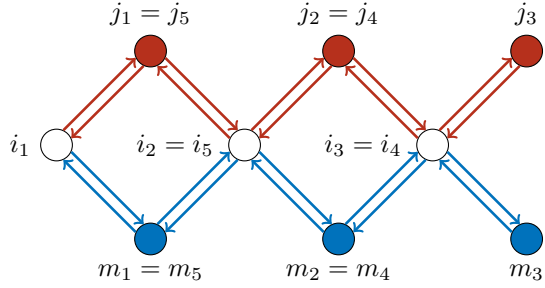

First of, we notice that each edge has to be crossed at least twice, we suppose the opposite, which means that there exist an edge $i_s m_s$ (without loss of generality), such that it is crossed only once, hence we get the following, by independence and centering of the entries of $X$,
\[
\Pi(G) 
= 
\E \left[ X_{IM} Y_{IJ} \right] 
= 
\E \left[ X_{i_s m_s} \right] \cdot \E \left[ X_{i_1 m_1} X_{i_2 m_1} \ldots X_{i_s m_{s-1}} X_{i_{s+1} m_s} \ldots X_{i_k m_k} X_{i_1 m_k} Y_{IJ} \right] 
= 
0.
\]

In addition, the graph made by $W_x$ has $2k$ steps, which implies a maximum of $k$ edges. Plus, it is a connected graph, so we have $\# \text{vertices} \leq \# \text{edges} + 1.$ Actually, each edge has to be crossed exactly 2 times. If not, we would have a maximum of $k - 1$ edges for the path $W_x$ so a maximum of $k$ vertices. Suppose that $\mathcal{G'} \vcentcolon = \{ G \in \mathcal{G}:$ each edge is crossed at least 2 times and that there exist at least 1 edge that is crossed more than 2 times$\}$, suppose also that $s + 1 = \#$ vertices $i$ so that we have a maximum of $k-s-1$ vertices $m$ and $k-s-1$ vertices $j$, then:
\begin{align*}
\frac{1}{n d^k p^k}& \sum_{G \in \mathcal{G'}} \Pi(G) \cdot \# \{(W_x,W_y): G \text{ has the shape of } (W_x,W_y)\} 
\\&
\leq \frac{1}{n d^k p^k} \sum_{G \in \mathcal{G'}} \max \{ \Pi(G): G \in \mathcal{G'} \} \cdot \# \{(W_x,W_y): G \text{ has the shape of } (W_x,W_y)\}\\
&\leq
\frac{1}{n d^k p^k} \cdot \max \{ \Pi(G): G \in \mathcal{G'} \} \sum_{G \in \mathcal{G'}} n^{s+1} d^{k-s-1} p^{k-s-1} 
\\&=
\frac{1}{n} \cdot \frac{n^{s+1} d^{k-s-1} p^{k-s-1}}{d^k p^k} \cdot \max \{ \Pi(G): G \in \mathcal{G'} \} \sum_{G \in \mathcal{G'}} 1 
\leq 
\frac{1}{n} \cdot \frac{n^{s+1}}{d^{s+1} p^{s+1}} \cdot C \xrightarrow[n \shortrightarrow \infty]{} 0 \cdot \gamma^{s+1} \cdot C 
= 
0 
\end{align*}
\noindent with $C = \max \{ \Pi(G): G \in \mathcal{G'} \} \sum_{G \in \mathcal{G'}} 1$. Note that we have $C < \infty$ since $\mathcal{G'}$ depends only on $k$ and because of the hypothesis $\E [X_{ij}^k] < \infty, \E [Y_{i\ell}^k] < \infty$. Thus, to get a non-vanishing contribution to the moment, each edge must be crossed exactly 2 times.

Moreover, we notice that the shape of the path made by $W_y$ is a symmetry of the shape of the path made by $W_x$, hence we have $\# \text{vertices } j = \# \text{vertices } m$. To see this, first see that the two paths completely share the $i$ vertices so that there cannot be any discrepancies on this part of the path. Now, at any point of the path made by $W_x$ or $W_y$, any vertex in the graph formed by $W_x$ or $W_y$ has a maximum of 2 neighboring edges such that those edges have been crossed once, if not we would have a cycle which contradicts the fact that $W_x$ and $W_y$ are both a tree. To prove that the path made by $W_y$ is a symmetry of the shape of the path made by $W_x$, we do it by contradiction, we suppose that they are not a symmetry and illustrate this through an example. We start by having the following graph.

\begin{center}

\begin{tikzpicture}[node distance = 80 ,main/.style = {draw, circle, minimum size=12,}]

\node[main, label = left:{$i_1$},] (1) at (0,0) {};
\node[main, label = below:{$m_1$}, fill = RoyalBlue] (2) at (1.25, -1.25) {};
\node[main, label = above:{$j_1$}, fill = BrickRed] (3) at (1.25, 1.25) {};
\node[main, label = left:{$i_2$}] (4) at (2.5, 0) {};

\draw[->, RoyalBlue, line width = .1em] (1) -- (2);
\draw[->, RoyalBlue, line width = .1em] (2) -- (4);

\draw[->, BrickRed, line width = .1em] (1) -- (3);
\draw[->, BrickRed, line width = .1em] (3) -- (4);

\end{tikzpicture}

\end{center}

Since we suppose that $W_x$ and $W_y$ are not the same path, we see that $W_x$ will, along the path, consider a new vertex (here $m_2$ on the figure) but not $W_y$ or vice versa. This gives the graph

\begin{center}

\begin{tikzpicture}[node distance = 80 ,main/.style = {draw, circle, minimum size=12,}]

\node[main, label = left:{$i_1$},] (1) at (0,0) {};
\node[main, label = below:{$m_1$}, fill = RoyalBlue] (2) at (1.25, -1.25) {};
\node[main, label = above:{$j_1=j_2$}, fill = BrickRed] (3) at (1.25, 1.25) {};
\node[main, label = left:{$i_2$}] (4) at (2.5, 0) {};
\node[main, label = below:{$m_2$}, fill = RoyalBlue] (7) at (3.75, -1.25) {};

\draw[->, RoyalBlue, line width = .1em] (1) -- (2);
\draw[->, RoyalBlue, line width = .1em] (2) -- (4);
\draw[->, RoyalBlue, line width = .1em] (4.-45) -- (7.135);

\draw[->, BrickRed, line width = .1em] (1) -- (3);
\draw[->, BrickRed, line width = .1em] (3.-25) -- (4.115);
\draw[->, BrickRed, line width = .1em] (4.155) -- (3.-65);

\end{tikzpicture}

\end{center}

Notice that $i_2$ in the path $W_x$ has 2 neighboring edges such that those edges have been crossed once, in our case it's $i_2 m_1$ and $i_2 m_2$, but in $W_y$'s case, there is none. So when the path $W_x$ go back to $i_2$ (which will happen since $i_2 m_1$ and $i_2 m_2$ have to be crossed twice), the only way for $W_y$ to do that is through vertex $j_1$ or $j_2$ since they are the only connection to $i_2$, hence an edge will be crossed at least 3 times, which is a contradiction with the fact that each edge must be crossed twice.

The problem consists now in counting double trees which is exactly the same counting problem as the moment of Marchenko--Pastur \cite{baisilver}*{Section 3.1.1}. Remember that the path made by $W_y$ is the same path with the number of vertices $j$ equal to the number of the vertices $m$. Since each edge has been crossed exactly 2 times then the path made by $W_x$ will have $k$ edges and $k+1$ vertices. Suppose that $s+1 = \#$ vertices $i$ and $k-s = \#$ vertices $m$ = $\#$ vertices $j$. Therefore we get the following:
\begin{align*}
\mu_n^k
&= 
\frac{1}{n d^k p^k} \sum_{G \in \mathcal{G}} \Pi(G) \cdot \# \{(W_x,W_y): G\text{ has the shape of } (W_x,W_y)\}
\\&=
\sum_{s=0}^{k-1} \frac{n(n-1) \ldots (n-s) d(d-1) \ldots (d-k+s+1) p(p-1) \ldots (p-k+s+1)}{n d^k p^k} \frac{\sigma_x^{2k} \sigma_y^{2k}}{s+1} {\binom{k-1}{s}} {\binom{k}{s}}
\\& \xrightarrow[n \shortrightarrow \infty]{} 
\sum_{s=0}^{k-1} \gamma^s \frac{ \left( \sigma_x \sigma_y \right)^{2k}}{s+1} {\binom{k-1}{s}} {\binom{k}{s}}.
\end{align*}

\end{proof}

\begin{prop}\label{prop2}

Suppose that we have the same assumptions of Proposition \ref{prop1}, then

\[
\mathrm{Var} \left[ \int x^{k}\, d\mu_n \right] = O\left(\frac{1}{n^2}\right)
\]

\end{prop}

\begin{proof}

Following the same steps as the previous calculus and using the definition of the variance, we get the following:
\begin{align*}
&\Var \left[ \int x^k d\mu_n \right]\\
&=
\frac{1}{n^2 d^{2k} p^{2k}} \Var \sum_{\substack{1 \leq i_1, ..., i_k \leq n \\ 1 \leq m_1, ..., m_k \leq d \\ 1 \leq j_1, ..., j_k \leq p}} X_{i_1 m_1} X_{i_2 m_1} X_{i_2 m_2} X_{i_3 m_2} \ldots X_{i_k m_k} X_{i_1 m_k} Y_{i_1 j_1} Y_{i_2 j_1} Y_{i_2 j_2} Y_{i_3 j_2} \ldots Y_{i_k j_k} Y_{i_1 j_k} 
\\&=
\frac{1}{n^2 d^{2k} p^{2k}} \sum_{\substack{I,I' \in [n]^k \\ M,M' \in [d]^k \\ J,J' \in [p]^k}} \E \left[ X_{IM} Y_{IJ} X_{I'M'} Y_{I'J'} \right] - \E \left[ X_{IM} Y_{IJ} \right] \E \left[ X_{I'M'} Y_{I'J'} \right]
\\&=
\frac{1}{n^2 d^{2k} p^{2k}} \sum_{G \in \mathcal{G}} \Pi(G) \cdot \# \{(W_x, W_{x'}, W_y, W_{y'}): G \text{ has the shape of } (W_x, W_{x'}, W_y, W_{y'})\}.
\end{align*}

\noindent We transform again the problem into a graph problem with $\mathcal{H}$ being the set of the all the graphs $G$ with 8$k$ steps that have the same shape of the paths $\left( W_x, W_{x'}, W_y, W_{y'} \right)$, and
$$\Pi(G) \vcentcolon = \E \left[ X_{IM} Y_{IJ} X_{I'M'} Y_{I'J'} \right] - \E \left[ X_{IM} Y_{IJ} \right] \E \left[ X_{I'M'} Y_{I'J'} \right].$$ In fact we are going to focus on the graph made by the paths $(W_x, W_{x'})$ because $\#$ vertices $M$ =  $\#$ vertices $J$ and $\#$ vertices $M'$ =  $\#$ vertices $J'$. Suppose that $G_{xx'}$ is the graph made by $(W_x, W_{x'})$ and that $G_x$ and $G_{x'}$ are the graphs made respectively by $W_x$ and $W_{x'}$. First of, notice that $G_{xx'}$ has 4$k$ steps. Furthermore, each edge has to be crossed at least twice, if not, $\Pi(G) = 0$, hence we have a maximum of $2k$ edges. Moreover, it has to be a connected graph, if not:

\begin{multline*}
\Pi(G) 
=
\E \left[ X_{IM} Y_{IJ} X_{I'M'} Y_{I'J'} \right] - \E \left[ X_{IM} Y_{IJ} \right] \E \left[ X_{I'M'} Y_{I'J'} \right]
\\=
\E \left[ X_{IM} Y_{IJ} \right] \E \left[ X_{I'M'} Y_{I'J'} \right] - \E \left[ X_{IM} Y_{IJ} \right] \E \left[ X_{I'M'} Y_{I'J'} \right] = 0.
\end{multline*}

\noindent Since $G_{xx'}$ is a connected graph then $\# \text{ vertices} \leq \# \text{ edges} + 1$, hence $|G_{xx'}| \leq 2k + 1$ with $|G_{xx'}|$ the number of vertices of the graph $G_{xx'}$. But if we have $|G_{xx'}| = 2k + 1$ then $G_{xx'}$ is a tree. And since $W_x$ starts and finishes with the same vertice $i_1$ then $G_x$ is also a tree and every edge of $G_x$ is crossed at least twice by $G_x$. Using the same logic $G_{x'}$ is a tree as well and every edge of $G_{x'}$ is crossed at least twice by $G_{x'}$. Moreover we should have an edge in common between the paths $W_{x}$ and $W_{x'}$, if not, we would have independence and $\Pi(G) = 0$ just like before. Hence we have an edge in common and it has to be crossed at least once by both $G_x$ and $G_{x'}$, so the edge in common is crossed at least 4 times by $G_x$ and $G_{x'}$ which implies that $|G_{xx'}| < 2k + 1$, hence we have that $|G_{xx'}| \leq 2k$. Suppose that $s + 1 = \#$ vertices $i$, and $2k-s-1 = \#$ vertices $M$ + $\#$ vertices $M' = \#$ vertices $J + \#$ vertices $J'$. Then we get the following:
\begin{align*}
\mathrm{Var} \left[ \int x^k d\mu_n \right]
&=
\frac{1}{n^2 d^{2k} p^{2k}} \sum_{G \in \mathcal{H}} \Pi(G) \cdot \# \{(W_x, W_{x'}, W_y, W_{y'}): G \text{ has the shape of } (W_x, W_{x'}, W_y, W_{y'})\}
\\
&\leq
\frac{1}{n^2 d^{2k} p^{2k}} \sum_{G \in \mathcal{H}} \max \{ \Pi(G): G \in \mathcal{H} \} \cdot \frac{n!}{(n-s-1)!} \frac{d!}{(d-2k+s+1)!}\frac{p!}{(p-2k+s+1)!}
\end{align*}
\noindent using the fact that $\# \{(W_x, W_{x'}, W_y, W_{y'}): G \text{ has the shape of } (W_x, W_{x'}, W_y, W_{y'})\} \leq n(n-1) \ldots (n-s) d(d-1) \ldots (d-2k+s+2) p(p-1) \ldots (p-2k+s+2)$, hence:
\begin{align*}
\mathrm{Var} \left[ \int x^k d\mu_n \right]
&\leq
\frac{1}{n^2 d^{2k} p^{2k}} \sum_{G \in \mathcal{H}} \max \{ \Pi(G): G \in \mathcal{H} \} n^{s+1} d^{2k-s-1} p^{2k-s-1}
\\& =
\frac{n^{s+1} d^{2k-s-1} p^{2k-s-1}}{n^2 d^{2k} p^{2k}} \max \{ \Pi(G): G \in \mathcal{H} \} \sum_{G \in \mathcal{H}} 1
=
\frac{1}{n^2} \cdot \frac{n^{s+1}}{d^{s+1} p^{s+1}} \cdot C = O\left(\frac{1}{n^2}\right)
\end{align*}

\noindent with $C = \max \{ \Pi(G): G \in \mathcal{G} \} \sum_{G \in \mathcal{H}} 1 < \infty$.

\end{proof}

We now have convergence of the expected moments as well as concentration under the form of a bound of the variance of the moments. However, these results hold under moment assumptions on the entries. We now work to remove the moment assumptions by comparing our initial distribution to a truncated one. To do so, we first introduce the L\'evy distance on cumulative distribution function, 
\[
\mathcal{L}(F,G) = \inf\{\varepsilon>0,\,G(x-\varepsilon)-\varepsilon\leqslant F(x)\leqslant G(x+\varepsilon)+\varepsilon,\,\forall x\in\mathbb{R}\}
\]
and we recall that convergence in this metric implies convergence in distribution.
\begin{prop}\label{prop3}

Suppose that we have the same assumptions as Theorem \ref{theo}. We define for $c>0$,
\[
\tilde{M} = \frac{1}{d} \tilde{X} \tilde{X}^\top \odot \frac{1}{p} YY^\top\quad\text{with}\quad\tilde{X}_{ij} = X_{ij}\mathds{1}_{|X_{ij}| \leq c}-\E\left[X_{ij}\mathds{1}_{ |X_{ij}|\leq c}\right].
\] If we consider 
\[
F^M(x) = \frac{1}{n}\#\{i,\,\lambda_i\leqslant x\}\quad\text{and}\quad F^{\tilde{M}}(x) = \frac{1}{n}\#\{i,\tilde{\lambda}_i\leqslant x\}
\] where $\lambda_i$, $\tilde{\lambda}_i$ are the eigenvalues of $M$ and $\tilde{M}$ respectively, then:
\[
\lim_{c\to\infty} \lim_{n\to\infty} \E \left[ \mathcal{L}\left( F^M, F^{\tilde{M}} \right)^3 \right] = 0
\]

\end{prop}

\begin{proof}

By \cite{baisilver}*{Corollary A.41} we have:
\begin{align*}
\E \left[ \mathcal{L} \left( F^M, F^{\tilde{M}} \right)^3 \right] 
&\leq 
\frac{1}{n} \E \left[ \text{Tr} \left( \left( M - \tilde{M} \right) \left( M - \tilde{M} \right)^\top \right) \right]
\\&=
\frac{1}{n} \sum_{\substack{1 \leq i,j \leq n \\ i \neq j}} \E \left[ \left( m_{ij} - \tilde{m}_{ij} \right)^2 \right] + \frac{1}{n} \sum_{i=1}^n \E \left[ \left( m_{ii} - \tilde{m}_{ii} \right)^2 \right]
\end{align*}
\noindent we begin by calculating the part where $i \neq j$. Using the definition of $M$ and $\tilde{M}$ we get the following:

\[
\frac{1}{n} \sum_{\substack{1 \leq i,j \leq n \\ i \neq j}} \E \left[ \left( m_{ij} - \tilde{m}_{ij} \right)^2 \right]
=
\frac{1}{n} \sum_{\substack{1 \leq i,j \leq n \\ i \neq j}} \E \left[ \left( \frac{1}{d} \sum_{k=1}^{d} X_{ik}X_{jk} - \tilde{X}_{ik} \tilde{X}_{jk} \right)^2 \right] \E \left[ \left( \frac{1}{p} \sum_{l=1}^{p} Y_{il}Y_{jl} \right)^2 \right]
\]

\noindent and since we have by independence of the entries of $Y$,
\[
\E \left[ \left( \frac{1}{p} \sum_{l=1}^{p} Y_{il} Y_{jl} \right)^2 \right]
=
\frac{1}{p^2} \sum_{1 \leq l_1, l_2 \leq p} \E \left[ Y_{il_1} Y_{jl_1} Y_{il_2} Y_{jl_2} \right]
=
\frac{1}{p^2} \sum_{1 \leq l_1 \leq p} \E \left[ Y_{il_1}^2 \right] \E \left[ Y_{jl_1}^2 \right] = \frac{1}{p} {\sigma_y}^4.
\]
This gives
\[
\frac{1}{n} \sum_{\substack{1 \leq i,j \leq n \\ i \neq j}} \E \left[ \left( m_{ij} - \tilde{m}_{ij} \right)^2 \right]
=
\frac{{\sigma_y}^4}{npd^2} \sum_{\substack{1 \leq i,j \leq n \\ i \neq j}} \E \left[ \left( \sum_{k=1}^{d} X_{ik} X_{jk} - \tilde{X}_{ik} \tilde{X}_{jk} \right)^2 \right]
\]

\noindent which has the same form as $\left( \sum_{k} a_k \right)^2 = \sum_{k} a_k^2 + 2 \sum_{k_1 < k_2} a_{k_1} a_{k_2}$, however we have that for  $k_1<k_2$:
\[
\E \left[ \left( X_{ik_1} X_{jk_1} - \tilde{X}_{ik_1} \tilde{X}_{jk_1} \right) \left( X_{ik_2} X_{jk_2} - \tilde{X}_{ik_2} \tilde{X}_{jk_2} \right) \right]
=
\left( \E[X_{11}] \right)^4 - 2 \left( \E \left[ X_{11} \right] \right)^2 \left( \E \left[ \tilde{X}_{11} \right] \right)^2 + \left( \E \left[ \tilde{X}_{11} \right] \right)^4
=
0.
\]
\noindent And using the inequality $\left( a + b \right)^2 \leq 2 \left( a^2 + b^2 \right)$ we get:
\begin{align*}
\frac{1}{n} \sum_{\substack{1 \leq i,j \leq n \\ i \neq j}} \E \left[ \left( m_{ij} - \tilde{m}_{ij} \right)^2 \right]
&=
\frac{{\sigma_y}^4}{npd^2} \sum_{\substack{1 \leq i,j \leq n \\ i \neq j}} \sum_{k=1}^{d} \E \left[ \left( X_{ik} \left( X_{jk} - \tilde{X}_{jk} \right) + \tilde{X}_{jk} \left( X_{ik} - \tilde{X}_{ik} \right) \right)^2 \right]
\\ &\leq
\frac{{4 \sigma_y}^4}{npd^2} \sum_{\substack{1 \leq i,j \leq n \\ i \neq j}} \sum_{k=1}^{d} \max \left\{ \E \left[ X_{ik}^2 \right], \E \left[ \tilde{X}_{ik}^2 \right] \right\} \E \left[ \left( X_{jk} - \tilde{X}_{jk} \right)^2 \right]
\\ &\xrightarrow[n \shortrightarrow \infty]{}
4 \gamma \sigma_y^4 \max \left\{ \E \left[ X_{ik}^2 \right], \E \left[ \tilde{X}_{ik}^2 \right] \right\} \E \left[ \left( X_{jk} - \tilde{X}_{jk} \right)^2 \right] \xrightarrow[c \shortrightarrow \infty]{} 0
\end{align*}
where we use the fact that 
\[\E \left[ \left( X_{jk} - \tilde{X}_{jk} \right)^2 \right] = \E \left[ \left( X_{jk} - {X}_{jk} \mathds{1}_{ |{X}_{jk}| < c} + \E \left[ {X}_{jk} \mathds{1}_{|{X}_{jk}| < c } \right] \right)^2 \right] \xrightarrow[c \shortrightarrow \infty]{} 0.
\]

\noindent Now we want to calculate the part where $i = j$:
\begin{multline*}
\frac{1}{n} \sum_{i=1}^n \E \left[ \left( m_{ii} - \tilde{m}_{ii} \right)^2 \right]
=
\frac{1}{n} \sum_{i=1}^{n} \E \left[ \left( \frac{1}{d} \sum_{k=1}^{d} X_{ik}^2 - \tilde{X}_{ik}^2 \right)^2 \right] \E \left[ \left( \frac{1}{p} \sum_{l=1}^{p} Y_{il}^2 \right)^2 \right]
\\=
\frac{1}{n d^2 p^2} \sum_{i=1}^{n} \sum_{1 \leq k_1, k_2 \leq d} \E \left[ X_{ik_1}^2 X_{ik_2}^2 - X_{ik_1}^2 \tilde{X}_{ik_2}^2 - \tilde{X}_{ik_1}^2 X_{ik_2}^2 + \tilde{X}_{ik_1}^2 \tilde{X}_{ik_2}^2 \right] \E \left[ \sum_{1 \leq l_1, l_2 \leq p} Y_{il_1}^2 Y_{il_2}^2 \right]
\end{multline*}
\noindent we have $d^2 - d$ cases where $k_1 \neq k_2$ and $d$ cases where $k_1 = k_2$, similarly we have $p^2 - p$ cases where $l_1 \neq l_2$ and $p$ cases where $l_1 = l_2$, thus
\begin{multline*}
\frac{1}{n} \sum_{i=1}^n \E \left[ \left( m_{ii} - \tilde{m}_{ii} \right)^2 \right]
=
\frac{ \left(p^2 - p \right) \left( \E \left[ Y_{11}^2 \right] \right)^2 + p \E \left[ Y_{11}^4 \right]}{p^2}\times
\\
\times\frac{\left( d^2 - d \right) \left( \left(\E \left[ X_{11}^2 \right] \right)^2 - 2 \E \left[ X_{11}^2 \right] \E \left[ \tilde{X}_{11}^2 \right] + \left( \E \left[ \tilde{X}_{11}^2 \right] \right)^2 \right) + d \left( \E \left[ X_{11}^4 \right] - 2 \E \left[ X_{11}^2 \tilde{X}_{11}^2 \right] + \E \left[ \tilde{X}_{11}^4 \right] \right)}{d^2}
\\
\xrightarrow[n \shortrightarrow \infty]{} \sigma_y^4 \left( \left(\E \left[ X_{11}^2 \right] \right)^2 - 2 \E \left[ X_{11}^2 \right] \E \left[ \tilde{X}_{11}^2 \right] + \left( \E \left[ \tilde{X}_{11}^2 \right] \right)^2 \right)
\end{multline*}
and since we have
\[
\E \left[ \tilde{X}_{11}^2 \right]
=
\E \left[ \left(X_{11} \mathds{1}_{|X_{11}| < c} - \E \left[ X_{11} \mathds{1}_{ |X_{11}| < c } \right] \right)^2 \right]
\xrightarrow[c \shortrightarrow \infty]{}
\E \left[ \left(X_{11} - \E \left[ X_{11} \right] \right)^2 \right]
=
\E \left[ X_{11}^2 \right],
\]
we finish by
\[
\frac{1}{n} \sum_{i=1}^n \E \left[ \left( m_{ii} - \tilde{m}_{ii} \right)^2 \right]
\xrightarrow[n \shortrightarrow \infty]{}
\sigma_y^4 \left( \left(\E \left[ X_{11}^2 \right] \right)^2 - 2 \E \left[ X_{11}^2 \right] \E \left[ \tilde{X}_{11}^2 \right] + \left( \E \left[ \tilde{X}_{11}^2 \right] \right)^2 \right)
\xrightarrow[c \shortrightarrow \infty]{}
0.
\]

\end{proof}

We can now finish the proof of Theorem \ref{theo} by combining the previous results.

\begin{proof}[Proof of Theorem \ref{theo}]
We first note that from convergence of expected moments from Proposition \ref{prop1} and the summable bound on the variance of the moments from Proposition \ref{prop2}, we obtain the almost sure weak convergence of the distribution to Marchenko--Pastur (see \cite{baisilver} for instance) which gives the second part of the theorem.

For the first part, we use the centering and truncating from Proposition \ref{prop3}. We first define $\hat{M} = \frac{1}{d} \tilde{X} \tilde{X}^\top \odot \frac{1}{p} \tilde{Y} \tilde{Y}^\top$, with $\tilde{Y}_{ij} = Y_{ij} \mathds{1}_{|Y_{ij}| \leq c}-\E[Y_{ij}\mathds{1}_{ |Y_{ij}|\leq c}]$. Since $X$ and $Y$ are playing the same role in the definition of $M$, by doing the same steps as Proposition \ref{prop3} we have that $ \lim_{c\to\infty} \lim_{n\to\infty} \E \left[ \mathcal{L}\left( F^{\hat{M}},F^{\tilde{M}} \right)^3 \right] = 0$, hence by the two steps we get that
\[
\lim_{c\to\infty} \lim_{n\to\infty} \E \left[ \mathcal{L} \left( F^{\hat{M}},F^{M} \right)^3 \right]
=
0
\]
\noindent Since the entries of $\hat{M}$ are bounded by definition and centered, we can use Propositions \ref{prop1} and \ref{prop2} to see that the empirical eigenvalue distribution of $\hat{M}$ converges weakly almost surely to the Marchenko--Pastur distribution of shape $\gamma$. This $L^3$ bound of the L\'evy distance gives that it converges in probability which gives our convergence in distribution in probability of the empirical eigenvalue distribution.
\end{proof}


\bibliographystyle{abbrv}
\bibliography{bibliostage.bib}
\end{document}